\documentclass{amsart}
\usepackage{amssymb,latexsym}
\theoremstyle{plain}
\newtheorem{theorem}{Theorem}

\newtheorem{lemma}{Lemma}
\theoremstyle{definition}

\begin{document}

\title[real rank]{On the typical rank of real bivariate
polynomials}
\author{Edoardo Ballico}
\address{Dept. of Mathematics\\
  University of Trento\\
38123 Povo (TN), Italy}
\email{ballico@science.unitn.it}
\thanks{The author was partially supported by MIUR and GNSAGA of 
INdAM (Italy).}

\begin{abstract}
Here we study the typical rank for real bivariate homogeneous polynomials of degree $d\ge 6$
(the case $d\le 5$ being settled by P. Comon and G. Ottaviani).
We prove that $d-1$ is a typical rank and that if $d$ is odd, then $(d+3)/2$ is a typical rank.
\end{abstract}

\maketitle

\section{Introduction}\label{S1}

For any integer $d\ge 0$ let $\mathbb {K}[x,y]_d$, $\mathbb {K}$ either $\mathbb {C}$ or $\mathbb {R}$, denote the $(d+1)$-dimensional $\mathbb {K}$-vector space of all degree $d$ bivariate homogeneous polynomials. For any $f\in \mathbb {R}[x,y]_d$ (resp. $f \in \mathbb {C}[x,y]_d$) let
$Rsr (f)$ (resp. $Csr (f))$ denote the minimal integer $r$ such that $f = \sum _{i=1}^{d} c_iL_i^d$
with $c_i\in \mathbb {R}$ and $L_i\in \mathbb {R}[x,y]_1$ (resp. $c_i\in \mathbb {C}$ and $L_i\in \mathbb {C}[x,y]_1$). The positive integer $Rsr (f)$ (resp. $Csr (f)$) is called the real (resp.
complex) rank of $f$. If $ f\in \mathbb {R}[x,y]_d$, then both $Rsr (f)$ and $Csr (f)$ are defined.
Obviously $Rsr (f) \le Csr (f)$. Quite often strict inequality holds (see e.g. \cite{co} and references
therein). The computation of the integer $Rsr (f)$ is used in real-life applications (\cite{bcmt},\cite{l} and the introductions
of \cite{bgi}, \cite{lt}, \cite{t}). However,
in many cases coming from Engineering the coefficients of $f$ are known only approximatively.
Unfortunately, $Rrs (f)$ is neither upper semicontinuous not lower semicontinuous.
Over $\mathbb {C}$ it is known the existence of a non-empty Zariski open subset $\mathcal {U}$
of $\mathbb {C}[x,y]_d\setminus \{0\}$ such that $Csr (f) = \lfloor (d+2)/2\rfloor $ for every
$f\in \mathcal {U}$ (\cite{co}, \cite{ik}, \S I.3).  We recall that Zariski open implies that $\mathcal {U}$ is dense
in $\mathbb {C}[x,y]_d$ for the euclidean topology and that $\mathbb {C}[x,y]_d\setminus \mathcal {U}$ is a union of finitely many differentiable manifolds
with real codimension at least $2$. In the complex case much more is known, even for
$f\in \mathbb {C}[x,y]_d\setminus \mathcal {U}$ (\cite{ik}, \S 1.3, \cite{l}, 9.2.2, \cite{cs}, \cite{bgi}, \S 3, \cite{lt}, Theorem 4.1). In the real case the picture is more complicated, because
$\mathcal {U}\cap \mathbb {R}[x,y]_d$ may have several connected components. An integer $t>0$
is called a {\it typical rank} in degree $d$ if there is a non-empty open subset $V$ (for the euclidean topology) of the real vector space $\mathbb {R}[x,y]_d$ such that $Rsr (f)=t$ for all $f\in V$. The existence of $\mathcal {U}\subset \mathbb {C}[x,y]_d$ such that $Csr (g) =\lfloor (d+2)/2\rfloor$ for all $g\in \mathcal {U}$
implies that any typical rank is at least $\lfloor (d+2)/2\rfloor$. If $t$ is a typical rank in degree $d$, then $t\le d$ (\cite{co}, Proposition 2.1). It is well-known that $\lfloor (d+2)/2\rfloor$
and $d$ are typical ranks and there is a clear description of a euclidean open subset of
$\mathbb {R}[x,y]_d$ parametrizing polynomials with real rank $d$: the set of all polynomials
with $d$ distinct real roots (\cite{co}, Proposition 3.4, \cite{cr}, Corollary 1). P. Comon and G. Ottaviani found all typical ranks
for $d\le 5$ and conjectured that all integers $t$ such that $\lfloor (d+2)/2\rfloor \le t \le d$
are typical ranks for real bivariate forms of degree $d$. 

In this note we prove the following results.  

\begin{theorem}\label{i7}
For each $d\ge 5$ the integer $d-1$ is a typical rank for real bivariate degree $d$ forms.
\end{theorem}

\begin{theorem}\label{i0}
Fix an odd integer $d=2m+1\ge 5$. Then $m+1$, $m+2$ and $2m+1$ are typical ranks for
real degree $d$ bivariate forms.
\end{theorem}

It is well-known that $\lfloor (d+2)/2\rfloor$ is always a typical rank. This observation, \cite{co} and
Theorems \ref{i7}, \ref{i0} prove Comon-Ottaviani conjecture if $d\le 7$, that for any even $d \ge 6$
there are at least $3$ typical ranks and that for every odd $d\ge 7$ there are at least $4$ typical ranks.

\section{The proofs}

For any $f\in \mathbb {R}[x,y]_d\setminus \{0\}$ and any $c\in \mathbb {R}\setminus \{0\}$ we
have $Rsr (f) = Rsr (cf)$. Hence the question about the real rank is a question
concerning polynomials, up to a non-zero scalar multiple. Hence we work
with the projective space $\mathbb {P}(\mathbb {R}[x,y]_d) \cong \mathbb {P}^d(\mathbb {R})$. Let $\nu _d: \mathbb {P}^1(\mathbb {C})\to \mathbb {P}^d(\mathbb {C})$ denote the degree $d$ Veronese embedding defined over $\mathbb {R}$. Set $C_r(\mathbb {C}) =\nu _d(\mathbb {P}^1(\mathbb {C}))$ (the degree $d$ rational normal curve) and $C_r(\mathbb {R}):= \nu _d(\mathbb {P}^1(\mathbb {R}))$.
Abusing notations, let $\sigma : \mathbb {P}^1(\mathbb {C})\to \mathbb {P}^1(\mathbb {C})$ and $\sigma : \mathbb {P}^d(\mathbb {C})\to \mathbb {P}^d(\mathbb {C})$ denote the complex conjugation.
We have $\sigma \circ \nu _d = \nu _d\circ \sigma$, $\mathbb {P}^1(\mathbb {R}) =\{P\in \mathbb {P}^1(\mathbb {C}):\sigma (P) =P\}$ and $\mathbb {P}^d(\mathbb {R}) =\{P\in \mathbb {P}^d(\mathbb {C}):\sigma (P) =P\}$.
Fix $f\in \mathbb {C}[x,y]\setminus \{0\}$ and call $P\in \mathbb {P}^d(\mathbb {C})$ the
associated point. We have $Csr (f)=t$ if and only if
$t$ is the minimal cardinality of a finite set $S\subset C_r(\mathbb {C})$ such
that $P\in \langle S\rangle$, where $\langle \ \ \rangle$ denote the linear span with complex coefficients.
If $f$ has real coefficients, then $Rsr (f)$ is the minimal cardinality of a finite set $S\subset C_r(\mathbb {R})$ such
that $P\in \langle S\rangle _{\mathbb {R}}$, where $\langle \ \ \rangle _{\mathbb {R}}$ denote the linear span with real coefficients. Notice that in the last definition
we may take $\langle S\rangle$ instead
of  $\langle S\rangle _{\mathbb {R}}$, because $\langle A\rangle \cap \mathbb {P}^d(\mathbb {R})
= \langle A\rangle  _{\mathbb {R}}$ for any finite set $A\subset \mathbb {P}^d(\mathbb {R})$.

\begin{lemma}\label{e1}
Fix $P\in \mathbb {P}^r(\mathbb {C})$ and assume the existence of finite sets $A, B\subset C_r(\mathbb {C})$ such that $P\in \langle A\rangle \cap \langle B\rangle$, $P\notin \langle A'\rangle $
for any $A'\subsetneq A$ and $P\notin \langle B'\rangle $
for any $B'\subsetneq B$. Then either $A=B$ or $\sharp (A\cup B) \ge r+2$.
\end{lemma}

\begin{proof}
Assume $A\ne B$. By \cite{bb}, Lemma 1, we have $h^1(C_r(\mathbb {C}),\mathcal {I}_{A\cup B}(1))>0$.
Since $C_r(\mathbb {C})$ is a degree $r$ rational normal curve and $h^1(\mathbb {P}^1(\mathbb {C}),R)=0$ for
every line bundle $R$ on $\mathbb {P}^1(\mathbb {C})$ with degree $\ge -1$, we get
$\sharp (A\cup B) \ge \deg (C_r(\mathbb {C}))+2=r+2$.
\end{proof}

\vspace{0.3cm}

\qquad {\emph {Proof of Theorem \ref{i0}}} If $d=5$ the result is true by \cite{co}, ManTheorem (ii); P. Comon and G. Ottaviani gave an explicit determination of the integer $Rsr (f)$ for a sufficiently general $f$ (\cite{co}, \S 4). For any $d\ge 3$ the integer $d$ is a typical rank  for $\mathbb {R}[x,y]_d$ and it is associated to $f\in \mathbb {R}[x,y]_d$
with $d$ distinct real roots (\cite{co}, Proposition 3.1, \cite{cr}, Corollary 1). 

\quad (a) It is well-known that $m+1$ is a typical rank. We may check this observation in the following way. Fix any $S\subset C_d(\mathbb {R})$ such that
$\sharp (S) =m+1$. Since $C_d$ is a rational normal curve, $S$ is linearly independent, i.e.
$\dim (\langle S\rangle )=m$.
Fix any $P\in \langle S\rangle _{\mathbb {R}}$ such that $P\notin \langle S'\rangle _{\mathbb {R}}$
for any $S'\subsetneq S$. Since $P\in \mathbb {P}^d(\mathbb {R})$, the latter condition
is equivalent to $P\notin \langle S'\rangle$ for any $S'\subsetneq S$. Assume
$Csr (P)\le m$ and take $B$ evincing $Csr (P)$. Taking $r:= d$ and $A:= S$ in Lemma \ref{e1} we get
a contradiction. Hence $Csr (P)=m+1$.
Since $P\in \langle S\rangle _{\mathbb {R}}$ and $S\subset C_d(\mathbb {R})$,
we have $Rsr (P)\le \sharp (S)$. Hence $Rsr (P)=m+1$. 

\quad (b) Now we prove that
$m+2$ is a typical rank. Fix any finite set $W\subset C_d(\mathbb {C})$ such that
$\sigma (W)=W$, $\sharp (W)=m+1$ and at least one point of $W$ is not real. Fix
any $S\subset C_d(\mathbb {R})$ such that $\sharp (S)=m+2$ and $W\cap S=\emptyset$.
Since $\sharp (W\cup S) =d+2$ and $C_d(\mathbb {C})$ is a rational normal curve, $W\cup S$ spans $\mathbb {P}^d(\mathbb {R})$ and $\mathbb {P}^d((\mathbb {C})$. Since
$\dim (\langle W))=\sharp (W)-1$ and $\dim (\langle S \rangle ) =\sharp (S)$,
Grassmann's formula gives that $\langle W\rangle \cap \langle S\rangle$ is a point, $Q$.
We claim that $Rsr (Q)=m+2$. Fix any $W'\subset W$ such that $\sharp (W')=m$ (we do not assume
$\sigma (W')=W'$). Since $\sharp (W'\cup S)=d+1$ and $C_d(\mathbb {C})$ is a rational normal curve,
$W'\cup S$ spans $\mathbb {P}^d(\mathbb {C})$. Hence Grassmann's formula gives
$\langle W'\rangle \cap \langle S\rangle =\emptyset$. Hence $Q\notin \langle W'\rangle$
for any $W'\subsetneq W$. Hence $Csr (Q)=m+1$ and $W$ is the unique subset of
$C_d(\mathbb {C})$ evincing $Csr (Q)$ (\cite{ik}, Theorem 1.40; alternatively, apply Lemma \ref{e1} as in step (a)). Hence $Rsr (Q)>m+1$.
Since $P\in \langle S\rangle _{\mathbb {R}}$ and $\sharp (S)=m+2$, we get
$Rsr (Q)=m+2$. We need to check that, varying $W$ and $S$ the points $\langle W\rangle \cap
\langle S\rangle$ cover a non-empty open subset of $\mathbb {P}^d(\mathbb {R})$ for
the euclidean topology.

For any $W$ as above and any $O\in \langle W\rangle _{\mathbb {R}}\setminus
(\cup _{W'\subsetneq W}\langle W'\rangle )$ let $\mathcal {R}(O)$
denote the set of all $S\subset C_d(\mathbb {R})$ such that $S\cap W =\emptyset$, $\sharp (S) =m+2$ and $O\in \langle S\rangle$. As in step (a) we see that $Crs (O)=m+1$. Obviously $\cup _{W}( \langle W\rangle _{\mathbb {R}}\setminus
(\cup _{W'\subsetneq W}\langle W'\rangle )$ covers a  non-empty open subset of $\mathbb {P}^d(\mathbb {R})$ for
the euclidean topology. The set of all $S\in \mathcal {R}(O)$ depends from $m+2$ parameters,
the set of all pairs $\{(S,P)\}$, $S\in \mathcal {R}(O)$, $P\in \langle S\rangle _{\mathbb {R}}$ depends from $2m+3$
parameters. Hence is sufficient to prove that for any $W$ are above
and any $O\in \langle W\rangle )_{\mathbb {R}}\setminus
(\cup _{W'\subsetneq W}\langle W'\rangle )$ the set $\mathcal {R}(O)$ has at most dimension $2$.
Assume that this is not true, i.e. assume $\dim (\mathcal {R}(O)) \ge 3$. Hence $\dim (\mathcal {C}(O)) \ge 3$, where $\mathcal {C}(O)$
denote the set of all $S\subset C_d(\mathbb {C})$ such that $S\cap W =\emptyset$, $\sharp (S) =m+2$ and $O\in \langle S\rangle$. Fix a  general $(O_1,O_2)\in C_d(\mathbb {C})\times C_d(\mathbb {C})$.
By assumption there are at least two distinct $S_1, S_2\in \mathcal {C}(O)$ containing
$\{O_1,O_2\}$. Set $B_i:= S_i\setminus \{O_1,O_2\}$. Set $L:= \langle \{O_1,O_2\}\rangle$.
Let $\ell _L: \mathbb {P}^d(\mathbb {C})\setminus L \to \mathbb {P}^{d-2}(\mathbb {C})$ denote
the linear projection from the line $L$. Set $A_i:= \ell _L(B_i)$, $i=1,2$. Since $\ell _L \vert
C_d(\mathbb {C})\setminus \{O_1,O_2\}$ is injective, we have $A_1\ne A_2$. Since
$O\in \langle S_i\rangle$ and $O\notin \langle S'_i\rangle$ for
any $S'_i\subsetneq S_i$, we have $\ell _L(O)\in \langle A_i\rangle $ and $\ell _L(O)\notin \langle
A'_i\rangle$ for any $A'\subsetneq A_i$. Since $A_1\ne A_2$ and $\sharp (A_i) \le m$,
Lemma \ref{e1} gives a contradiction.\qed

\vspace{0.3cm}

\qquad {\emph {Proof of Theorem \ref{i7}.}}

Fix $f\in R[x,y]_{d-1}$ with $d-1$ distinct roots and write $f = \sum _{i=1}^{d-1} c_iL_i^{d-1}$
with $L_i\in \mathbb {R}[x,y]_1$, $c_i\in \mathbb {R}$, $c_i\ne 0$, and $L_i$, $L_j$
pairwise non-proportional for all $i\ne j$ (\cite{co}, Proposition 3.1, and \cite{cr}, Corollary 1). Set $g_{f,c, R}:= \sum _{i=1}^{d} c_iL_i^d + cR^d$, with $c \in \mathbb {R}$,
$R\in \mathbb {R}[x,y]_1$
and $R\ne 0$. Let $P_{f,c ,R}\in \mathbb {P}^d(\mathbb {R})$ be the point corresponding to the polynomial $g_{f,c, R}$. Let $O\in \mathbb {P}^1(\mathbb {R})$ be the point
associated to $R$. Set $Q:= \nu _d(O)\in C_d(\mathbb {R})$. Let $\ell _Q: \mathbb {P}^d(\mathbb {R})\setminus \{Q\} \to \mathbb {P}^{d-1}(\mathbb {R})$ denote
the linear projection from $Q$. We have $Rsr (\ell _Q(P_{f,c ,R}))
= Rsr (f) =d-1$. Assume $Rsr (P_{f,c,r}) \le d-2$ and take $S_1\subset C_d(\mathbb {R})$
evincing $Rsr (P_{f,c,r})$. If $Q\notin S_1$, then we get $\ell _Q(P_{f,c ,R}) \in \langle \ell _Q(S_1)\rangle _{\mathbb {R}}$ and hence $Rsr (\ell _Q(P_{f,c ,R})) =\sharp (S_1)\le  d-2$, a contradiction. If $Q\in S_1$, then we
get $\ell _Q(P_{f,c ,R}) \in \langle \ell _Q(S_1\setminus \{Q\})\rangle _{\mathbb {R}}$ and hence $Rsr (\ell _Q((P_{f,c ,R}) \le d-3$, a contradiction. Hence $Rsr (P_{f,c,R}) \ge d-1$.

\quad (a) In this step we check that varying $f$ and $c$ the points $P_{f,c,R}$ covers a non-empty open subset
of $\mathbb {P}^{d}(\mathbb {R})$ for the euclidean topology. We fix $R$ and hence $O$ and $Q =\nu _d(O)$ and take all real polynomials $f$ as above with the additional
condition that they have a representative $\sum _{i=1}^{d-1} c_iL_i^{d-1}$ with
no $L_i$ proportional to $R$. We cover in this way a non-empty open subset $U$ of $\mathbb {P}^{d-1}(\mathbb {R})$. We identify $\mathbb {P}^{d-1}(\mathbb {R})$ with a hyperplane $M(\mathbb {R})$ of $\mathbb {P}^{d}(\mathbb {R})$ not containing $Q$ and call $M(\mathbb {C})\subset \mathbb {P}^d(\mathbb {C})$ the corresponding complex hyperplane. We see $\ell _Q$ as a submersion
$\ell _Q: \mathbb {P}^d(\mathbb {C}) \setminus \{Q\} \to M(\mathbb {C})\subset \mathbb {P}^d(\mathbb {C})$. Since
its fibers are the lines through $Q$ (minus the point $\{Q\}$), $\ell _Q^{-1}(U)$ is a non-empty
euclidean open subset of $\mathbb {P}^{d}(\mathbb {R})$. Moreover, we may take a very
large open set, i.e. $\ell _Q^{-1}(U)$
as this open subset.

\quad (b) In this step we prove the existence of a non-empty open subset $V$ of $\mathbb {P}^{d}(\mathbb {R})$ corresponding to points $P_{f,c,R}$ associated to polynomials $g_{f,c,R}$ with distinct roots, not all of them real. 
Take the set-up of step (a). Fix $f\in U$. Notice that the set of all $g\in \mathbb {C}[x,y]_d$ with at least one multiple root
is an an algebraic subvariety $\Sigma$ of dimension $\le d-1$ and that the closure $\overline{\Sigma}$ of $\ell _Q(\Sigma \setminus \{O\})$ is a proper subvariety of $M(\mathbb {C})$. Hence
it is not dense in $M(\mathbb {C})$ for the euclidean topology. Hence there is a non-empty open subset $U_1$ of $U$ such that $\ell _Q^{-1}(U_1)\cap \Sigma =\emptyset$. Fix
$f\in U_1$. By construction for each $c\ne 0$ the polynomial $g_{f,c,R}$ has no multiple root.
To conclude the proof it is sufficient to prove the existence of a non-empty interval $J \subset \mathbb {R}$ such
that $g_{f,c,R}$ has not $d$ real roots for $c\in J$.
Up to a projective change of coordinates we may assume $R =y$. Write
$\sum _{i=1}^{d-1} c_iL_i^d = \sum _{i=0}^{d} a_ix^iy^{d-i}$. Set $y=1$, $u(x):= \sum _{i=0}^{d}a_ix^i$ and $u_c(x) :=u(x) + c^d$. By construction $u_c(x)$ has degree $d$ and $d$ distinct roots. Fix a real numer $T>0$ such that $u(x)$ is monotone for $x\le -T$ and $x\ge T$.
Let $\eta$ be the maximum of $\vert u(x)\vert$ in the interval $[-T,T]$. If $\vert c\vert \gg 0$, say $\vert c^d\vert > \eta$, then $u_c$ has at most $2$ real roots.\qed

\providecommand{\bysame}{\leavevmode\hbox to3em{\hrulefill}\thinspace}

\end{document}